\documentclass[11pt]{amsart}

\usepackage{latexsym}
\usepackage{amsfonts}

\def\cal{\mathcal}

\newcommand{\field}[1]{\mathbb{#1}}
\newcommand{\C}{\field{C}}

\newcommand{\R}{\field{R}}

\newtheorem{defi}{Definition}[section]

\newtheorem{lem}[defi]{Lemma}
\newtheorem{theo}[defi]{Theorem}
\newtheorem{co}[defi]{Corollary}

\newtheorem{re}[defi]{Remark}

\thanks{The  author was partially supported by the NCN
 grant 2014-2017}

\subjclass{14 D 99, 14 R 99, 51 M 99}

\title[The Thom Conjecture for  proper polynomial mappings]{The Thom Conjecture for proper polynomial mappings} \makeatletter

\author{Zbigniew Jelonek}

\address[Z. Jelonek]{Instytut Matematyczny\\
Polska Akademia Nauk\\
\'Sniadeckich 8, 00-656 Warszawa, Poland}
\email{najelone@cyf-kr.edu.pl}

\date{\today}

\begin{document}

\maketitle

\begin{abstract}{ Let $f,g:X \to Y$ be continuous mappings. We say
that $f$ is topologically equivalent to $g$ if there exist
homeomorphisms $\Phi : X\to X$ and $\Psi: Y\to Y$ such that
$\Psi\circ f\circ \Phi=g.$ Let $X,Y$ be complex smooth irreducible
affine varieties. We show that every algebraic family $F: M\times
X\ni (m, x)\mapsto F(m, x)=f_m(x)\in Y$ of polynomial mappings
contains only a finite number of topologically non-equivalent
proper  mappings. In particular there are only a finite number of
topologically non-equivalent proper polynomial mappings $f:
\C^n\to\C^m$ of bounded (algebraic) degree. This gives a positive
answer to the Thom Conjecture in the case of proper polynomial
mappings.}
\end{abstract}

\section{Introduction}
Let $f,g:X \to Y$ be continuous mappings. We say that $f$ is {\it
topologically equivalent} to $g$ if there exist homeomorphisms
$\Phi : X\to X$ and $\Psi: Y\to Y$ such that $\Psi\circ f\circ
\Phi=g.$ In the case $X=\C^n$ and $Y=\C$ Rene Thom stated a
Conjecture that there are only finitely many topological types of
polynomials $f: X\to Y$ of bounded degree. This Conjecture was
confirmed by T. Fukuda  \cite{fuk}. Also a more general problem
was considered: how many topological types are there in the family
$P(n,m,k)$ of polynomial mapping  $f:\C^n\to\C^m$ of degree
bounded by $k$? K. Aoki and H. Noguchi \cite{a-n} showed  that
there are only a finite number of topologically non-equivalent
mappings in the family $P(2,2,k).$ Finally I. Nakai  \cite{nak}
showed that each familiy $P(n,m,k)$, where $n, m ,k > 3,$ contains
infinitely many different topological types. Hence the General
Thom Conjecture is not true. However, we show in this paper that
the General Thom Conjecture is  true in the following important
case: for every $n,m$ and $k$ there are only a finite number of
topological types of {\it proper} polynomial mappings $f:
\C^n\to\C^m$ of (algebraic) degree bounded by $k.$ In fact we
prove more: if $X,Y$ are smooth affine irreducible varieties, then
every algebraic family $\cal F$ of polynomial mappings from $X$ to
$Y$ contains only a finite number of topologically non-equivalent
proper mappings.

Our proof goes as follows. Let $M$ be a smooth affine irreducible
variety and let $\cal F$ be  a family of polynomial mappings
induced be a regular mapping $F: M\times X \to Y,$ i.e., ${\cal
F}:=\{f_m: X\ni x\mapsto F(m,x)\in Y, \ m\in M\}.$ Let us recall
that if $f: X\to Z$ is a generically finite polynomial mapping of
affine varieties, then the {\it bifurcation set} $B(f)$ of $f$ is
the set $\{ z\in Z: z\in Sing(Z) \ or\ \# f^{-1}(z)\not=\mu(f)\}$,
where $\mu(f)$ is the topological degree of $f.$  The set $B(f)$
is always closed in $Z.$   We show that there exists a Zariski
open, dense subset $U$ of $M$ such that

1)  for every $m\in U$ we have $\mu(f_m)=\mu({\cal F}),$ where we
treat $f_m$ as a mapping $f_m: X\to Z_m:=\overline{f_{m}(X)},$

2) for every $m_1, m_2\in U$ the pairs $(\overline{f_{m_1}(X)},
B(f_{m_1}))$ and $(\overline{f_{m_2}(X)}, B(f_{m_2}))$  are
equivalent via a homeomorphism, i.e., there is a homeomorphism
$\Psi : Y\to Y$ such that
$\Psi(\overline{f_{m_1}(X)})=\overline{f_{m_2}(X)}$ and
$\Psi(B(f_{m_1}))=B(f_{m_2}).$

 \vspace{3mm}

In particular the group $G=\pi_1(\overline{f_{m}(X)}\setminus
B(f_m))$ does not depend on $m\in U.$ Using  elementary facts from
the theory of topological coverings, we show that the number of
topological types of proper mappings in the family ${\cal F}_{|U}$
is bounded by the number of subgroups of $G$ of  index $\mu({\cal
F})$, hence it is finite. Then we conclude the proof by induction.
Finally, the case of arbitrary  $M$ can be  easily reduced to the
smooth, irreducible, affine case.

It is worth  noting that the real counterpart of our result is not
true. Indeed, Rene Thom  \cite{thom} found the following family
${\cal F}:=\{ f_m : \R^3\to\R^3\}_{m\in \R}$ of real polynomial
mappings:
$$X=[x(x^2+y^2-a^2)-2ayz]^2[(x+my)(x^2+y^2-a^2)-2a(y-mx)z]^2,$$
$$Y=x^2+y^2-a^2,$$
$$Z=z,$$
(here $a\in \R^*$ is a fixed constant and $m$ is a parameter). He
proved   that $f_{m_1}$ is not topologically equivalent to
$f_{m_2}$ for $m_1\not=m_2.$ It is easy to see that all mappings
in the  family $\cal F$ are proper.

\begin{re}
{\rm In this paper we use the term "polynomial mapping" for every
regular mapping $f: X\to Y$ of affine varieties.}
\end{re}

\section{Bifurcation set}
Let $X,Z$ be  affine  irreducible varieties of the same dimension
and assume that $X$ is smooth. Let $f: X\to Z$ be a dominant
polynomial mapping. It is well known that there is a Zariski open
non-empty subset $U$ of $Z$ such that for every $x_1,x_2\in U$ the
fibers $f^{-1}(x_1), f^{-1}(x_2)$ have the same number $\mu(f)$ of
points. We say that  $\mu(f)$ is the topological degree of $f.$
Recall the following (see \cite{jel}, \cite{jel1}):

\begin{defi}
{\rm Let $X,Z$ be as above and let $f : X \rightarrow Z$ be a
dominant polynomial mapping.  We say that $f$ {\it is finite at a
point} $z \in Z$ if there exists an open neighborhood $U$ of $z$
such that the mapping $ f_{|f^{-1}(U)} :f^{-1} (U)\rightarrow U$
is proper.}
\end{defi}

It is well-known that the set $S_f$ of points at which the mapping
$f$ is not finite is either empty or it is a hypersurface (see
\cite{jel}, \cite{jel1}). We say that  $S_f$ is {\it the set of
non-properness} of  $f.$

\begin{defi}
Let $X$ be a smooth affine $n$-dimensional variety and let $Z$ be
an affine variety of the same dimension. Let $f : X\to Z$ be a
generically finite dominant polynomial mapping of geometric degree
$\mu(f).$ The bifurcation set  of  $f$ is
$$B(f)=\{z\in  Z : z\in Sing(Z) \ or \ \# f^{-1}(z)\not=\mu(f)\}.$$
\end{defi}

\begin{re}
{\rm The same definition makes sense for those continuous mapping
$f: X\to Z$, for which we can define the topological degree
$\mu(f)$ and singularities of $Z.$ In particular if $Z_1, Z_2$ are
affine algebraic varieties, $f: X\to Z_1$ is a dominant polynomial
mapping and  $\Phi : Z_1\to Z_2$ is a homeomorphism which
preserves singularities, then we can define  $B(\Phi\circ f)$ as
$\Phi(B(f)).$ Moreover, the mapping $\Phi\circ f$ behaves
topologically
 as an analytic covering. We will use this facts in the proof
of Theorem \ref{gl1}.}
\end{re}

We have the following  theorem (see also \cite{j-k}):

\begin{theo}\label{jk}
Let $X,Z$ be   affine irreducible complex varieties of the same
dimension  and suppose $X$ is smooth. Let $f: X \to Z$ be a
polynomial dominant mapping. Then the set $B(f)$ is closed.
\end{theo}

\begin{proof}
Let us note that outside the set $S_f\cup Sing(Z)$ the mapping $f$
is a (ramified) analytic covering of degree $\mu(f).$ By  Lemma
\ref{l} below, if $z\not\in  Sing(Z)$ we have $\# f^{-1}(z)\le
\mu(f).$ Moreover, since $f$ is an analytic covering outside
$S_f\cup Sing(Z)$ we see that for $y \not\in S_f\cup Sing(Z)$ the
fiber $f^{-1}(z)$  has exactly $\mu(f)$ points counted with
multiplicity. Take $X_0:=X \setminus f^{-1}(Sing(Z)\cup S_f).$ If
$z\in K_0(f_{|X_0})$, {the set of critical values of $f_{|X_0},$}
then $\# f^{-1}(z)<\mu(f).$

Now let $z\in S_f\setminus Sing(Z).$ There are {two}
possibilities:

a)  $\# f^{-1}(z)=\infty.$

b)  $\# f^{-1}(z)<\infty.$

In  case b) let $U$ be an affine neighborhood of $z$ disjoint from
$Sing(z)$ over which the mapping $f$ is quasi-finite. Let
$V=f^{-1}(U).$ By the Zariski Main Theorem in the version given by
Grothendieck, there exists a normal variety $\overline{V}$ and a
finite mapping $\overline{f}:\overline{ V}\to U$ such that

1) $V\subset  \overline{V}$,

2) $\overline{f}_{|V}=f.$

Since  $y\in \overline{f}(\overline{V}\setminus V)$, it follows
from  Lemma \ref{l} below that $\# f^{-1}(z)<\mu(f).$
Consequently, if $z\in S_f,$ we have $\# f^{-1}(z)<\mu(f).$
Finally, we have $B(f)= K_0(f_{|X_0})\cup S_f\cup Sing(Z).$
However, the set $K_0(f_{|X_0})$ is closed in $Z\setminus (S_f\cup
Sing(Z)).$ Hence $B(f)$ is closed in $Z.$
\end{proof}

\begin{lem}\label{l}
Let $X,Z$ be affine normal varieties of the same dimension. Let
$f:X\to Z$ be a finite mapping. Then for every $z\in Z$ we have
$\# f^{-1}(z)\le\mu(f).$
\end{lem}

\begin{proof}
Let $\# f^{-1}(z)=\{ x_1, \ldots , x_r\}.$ We can choose a
function $h\in \C[X]$ which separates all $x_i$ (in particular we
can take as $h$ the equation of a general hyperplane section).
Since $f$ is finite, the minimal equation of $h$ over the field
$\C(Z)$ is of the form: $$T^s+a_1(f)T^{{s-1}}+\ldots+a_s(f)\in
f^*\C[Z][T],$$ where $ s\le \mu(f).$ If we substitute $f=z$ into
this equation we get {the} desired result.
\end{proof}

\section{Main result}
We start with the following:

\begin{lem}\label{ber}
Let $f: X^k\to Y^l$ be a dominant polynomial mapping of affine
irreducible varieties.  There exists a Zariski open non-empty
subset  $U\subset Y$ such that for $y\in U$ we have
$Sing(f^{-1}(y))=f^{-1}(y)\cap Sing(X).$
\end{lem}

\begin{proof}
We can assume that $Y$ is smooth. Since there exists a mapping
$\pi: Y^l\to \C^l$ which is generically etale, we can assume that
$Y=\C^l.$ Let us recall that if $Z$ is an algebraic variety, then
a point $z\in Z$ is smooth if and only if the local ring ${\cal
O}_z(Z)$ is regular, or equivalently  $\dim_\C {\mathfrak
m}/{\mathfrak m}^2=\dim Z,$ where $\mathfrak m$ denotes the
maximal ideal of ${\cal O}_z(Z)$.

Let $y=(y_1,...,y_l)\in \C^l$ be a sufficiently generic point.
Then by Sard's Theorem the fiber $Z=f^{-1}(y)$ is smooth outside
$Sing(X)$ and $\dim Z=\dim X - l=k-l.$ Note that the generic
(scheme-theoretic) fiber $F$ of $f$ is reduced. Indeed, this fiber
$F=Spec(\C(Y)\otimes_{\C[Y]} \C[X])$ is the spectrum of a
localization of $\C [X]$ and so a domain. Since we are in
characteristic zero, the reduced $\C(Y)$-algebra
$\C(Y)\otimes_{\C[Y]} \C[X]$ is necessarily geometrically reduced
(i.e. stays reduced after extending to an algebraic closure of
$\C(Y)$). Since  the property of fibres being geometrically
reduced is open on the base, i.e. on $Y$, thus the fibres over an
open subset  of $Y$ will be reduced. Consequently, there is a
Zariski open, non-empty subset $U\subset Y$ such that for $y\in U$
the fiber  $f^{-1}(y)$ is reduced. Hence we can assume that $Z$ is
reduced. It is enough to show that every point $z\in Z\cap
Sing(X)$ is singular on $Z.$

Assume that $z\in Z\cap Sing(X)$ is smooth on $Z.$ Let $f: X\to
\C^l$ be given as $f=(f_1,...,f_l)$, where $f_i\in \C[X].$ Then
${\cal O}_z(Z)={\cal O}_z(X)/(f_1-y_1,...,f_l-y_l).$ In particular
if $\mathfrak m'$ denotes the maximal ideal of ${\cal O}_z(Z)$ and
$\mathfrak m$ denotes the maximal ideal of ${\cal O}_z(X)$ then
$\mathfrak m'=\linebreak \mathfrak m/(f_1-y_1,...,f_l-y_l).$ Let
$\alpha_i$ denote the class of the polynomial $f_i-y_i$ in
${\mathfrak m}/{\mathfrak m}^2.$ Let us note that
\begin{equation}\label{eq}
{\mathfrak m'}/{\mathfrak m'}^2={\mathfrak m}/({\mathfrak
m}^2+(\alpha_1,...,\alpha_l)).
\end{equation}
Since the point $z$ is smooth on $Z$ we have $\dim_\C {\mathfrak
m'}/{\mathfrak m'}^2=\dim Z=\dim X-l.$ Take a basis
$\beta_1,...,\beta_{k-l}$ of the space $ {\mathfrak m'}/{\mathfrak
m'}^2$ and let $\overline{\beta_i}\in {\mathfrak m}/{\mathfrak
m}^2$ correspond to $\beta_i$ under the correspondence (\ref{eq}).
Note that the vectors
$\overline{\beta_1},...,\overline{\beta_{k-l}}, \alpha_1,...,
\alpha_l$ generate the space ${\mathfrak m}/{\mathfrak m}^2.$ This
means that $\dim_\C {\mathfrak m}/{\mathfrak m}^2\le k-l+l=k=\dim
X.$ Hence the point $z$ is smooth on $X$, a contradiction.
\end{proof}

 We have:

\begin{lem}\label{fiber}
Let $X, Y$ be  smooth complex irreducible algebraic varieties and
$f:X\to Y$ a regular dominant mapping. Let $N\subset W\subset X$
be closed subvarieties of $X.$ Then there exists a non-empty
Zariski open subset $U\subset Y$ such that for every $y_1, y_2\in
U$ the triples  $(f^{-1}(y_1), W\cap f^{-1}(y_1), N\cap
f^{-1}(y_1))$ and $(f^{-1}(y_2), W\cap f^{-1}(y_2), N\cap
f^{-1}(y_2))$ are homeomorphic.
\end{lem}

\begin{proof}
Let $X_1$ be an algebraic completion of $X$ and let $\overline{Y}$
be a smooth algebraic completion of $Y.$ Take
$X_1':=\overline{graph(f)}\subset X_1\times \overline{Y}$ and let
$X_2$ be a desingularization of $X_1'.$

We can assume that $X\subset X_2.$ We have an induced mapping
$\overline{f}: X_2\to  \overline{Y}$ such that
$\overline{f}_{|X}=f.$ Let $Z=X_2\setminus X.$ Denote by
$\overline{N}, \overline{W}$ the closures of $N$ and $W$ in $X_2.$
Let $\cal R=\{\overline{N}\cap Z, \overline{W}\cap Z,
\overline{N},  \overline{W}, Z\} ,$ a collection of algebraic
subvarieties of $X_2.$ There is a Whitney stratification ${\cal
S}$ of $X_2$ which is compatible with $\cal R.$

For any smooth strata $S_i\in \cal S$ let $B_i$ be the set of
critical values of the mapping {$\overline{f}_{|S_i}$} and denote
$B=\overline{ \bigcup B_i}.$ Take $X_3=X_2\setminus
\overline{f}^{-1}(B).$  The restriction of the stratification
$\cal S$ to $X_3$ gives a Whitney {stratification} which is
compatible with the family ${\cal R}':={\cal R}\cap X_3.$  We have
a proper mapping $f':=\overline{f}_{|X_3} : X_3\to
\overline{Y}\setminus B$ which is a submersion on each stratum. By
the Thom first isotopy theorem there is a trivialization of $f'$
which preserves the strata. It is an easy observation that this
trivialization gives a trivialization of the mapping $f:
X\setminus f^{-1}(B)\to Y\setminus B:=U.$ In particular the fibers
$f^{-1}(y_1)$ and $f^{-1}(y_2)$ are homeomorphic via a stratum
preserving homeomorphism. This means that the triples
$(f^{-1}(y_1), W\cap f^{-1}(y_1), N\cap f^{-1}(y_1))$ and
$(f^{-1}(y_2), W\cap f^{-1}(y_2),  N\cap f^{-1}(y_2))$ are
homeomorphic.
\end{proof}

\vspace{3mm}

We also need the following:

\newpage

\begin{defi}
Let $X,Y$ be smooth affine varieties. By a family of regular
mappings ${\cal F}_M(X,Y,F):={\cal F}$ we mean a regular mapping
$F: M\times X\to Y$, where $M$ is an algebraic variety. The
members of a family $\cal F$ are the mappings $f_m: X\ni x\to
F(m,x)\in Y.$ Let
$$G: M\times X\ni (m,x)\mapsto (m, F(m,x))\in Z=\overline{G(M\times X)}\subset M\times Y.$$
If $G$ is generically finite, then by the topological degree
$\mu({\cal F})$ we mean the number $\mu(G)$. Otherwise we put
$\mu({\cal F})=0.$
\end{defi}

Later we will sometimes  identify the mapping $f_m$ with the
mapping $G(m,\cdot)=(m, f_m): X\to m\times Y.$ The following lemma
is important:

\begin{lem}\label{gl'}
Let $X,Y$ be  smooth affine complex varieties. Let $M$ be a smooth
affine irreducible variety and let $\cal F$ be the family induced
by a mapping $F: M\times X \to Y,$ i.e., $\cal F=\{ f_m: X\ni
x\mapsto F(m,x)\in Y, \ m\in M\}.$ Assume that  $\mu({\cal F})>0.$
Take $Z=\overline{G(M\times X)}$ and put $Z_m= (m\times Y)\cap Z.$
 Then

1)  there is an open non-empty subset $U_1\subset M$ such that for
every $m\in U_1$ we have $\mu(f_m)=\mu({\cal F});$

2) there is a non-empty open subset $U_2\subset U_1$ such that for
every $m\in U_2$ we have $\overline{f_m(X)}=Z_m:=(m\times Y)\cap
Z$ and $B(f_m)=B(G)_m:=(m\times Y)\cap B(G);$

3) there is a non-empty open subset $U_3\subset U_2$ such that for
every $m_1, m_2\in U_3$ the pairs  $(\overline{f_{m_1}(X)},
B(f_{m_1}))$ and $(\overline{f_{m_2}(X)}, B(f_{m_2}))$  are
equivalent by means of a homeomorphism, i.e., there is a
homeomorphism $\Psi : Y\to Y$ such that
$\Psi(\overline{f_{m_1}(X)})=\overline{f_{m_2}(X)}$ and
$\Psi(B(f_{m_1}))=B(f_{m_2}).$
\end{lem}

\begin{proof}
1) Take $G: M\times X\ni (m,x)\mapsto (m,f(m,x))\in Z.$ We know by
Theorem \ref{jk} that the mapping $G': M\times X \ni (m,x)\mapsto
(m, F(m,x))\in Z$ has a constant number of points in the fibers
outside the bifurcation set $B(G)\subset Z.$ Take $U=Z\setminus
B(G).$ Let $\pi: Z \ni (m, y)\mapsto m\in M$ be the projection. We
show that the constructible set $\pi(U)$ is dense in $M.$ Indeed,
assume that $\overline{\pi(U)}=N$ is a proper subset of $M.$ Since
$U$ is dense in $Z$, we have $\pi(Z)\subset N$, i.e., $Z\subset
N\times Y.$ This is a contradiction. In particular the set
$\pi(U)$ is dense in $M$ and it contains a Zariski open, non-empty
subset $U_1\subset M.$ Of course $\mu(f_m)=\mu(\cal F)$ for $m\in
U_1.$

2) Consider the projection $\pi : Z\ni (m,y)\mapsto m\in M.$ As we
know from 1), the mapping $\pi$ is dominant. By a well known
result, after shrinking $U_1$ we can assume that every fiber $Z_m$
of $\pi$ ($m\in U_2\subset U_1$) is of pure dimension $d={\rm
dim}\ Z- {\rm dim}\ M={\rm dim}\ X.$ However, $Z_m
=\overline{f_m(X)}\cup B(G)_m.$ Generically the dimension of
$B(G)_m$ is less than $d.$ Hence if we possibly shrink $U_2$, we
get $Z_m=\overline{f_m(X)}$ for $m\in U_2.$ Moreover, by Lemma
\ref{ber} (after shrinking $U_2$ if necessary), we can assume that
$Sing(Z_m)=Sing(Z)_m:=(m\times Y)\cap Sing(Z).$ Now it is easy to
see that $B(f_m)=B(G)_m.$

3) We have $\overline{f_m(X)}=Z_m$ and $B(f_m)= B(G)_m$ for $m\in
U_2$. Now apply Lemma \ref{fiber} with $X=U_2\times Y,$
$W=(U_2\times Y)\cap Z$, $N=(U_2\times Y)\cap B(G)$ and $f:
U_1\times Y\ni (m,y)\mapsto m\in U_1.$
\end{proof}

Now we are ready to prove our main result:

\begin{theo}\label{gl1}
Let $X,Y$ be smooth affine irreducible varieties. Every algebraic
family $\cal F$ of polynomial mappings from $X$ to $Y$ contains
only a finite number of topologically non-equivalent proper
mappings.
\end{theo}

\begin{proof}
The proof is by induction on $\dim M.$ We can assume that $M$ is
affine, irreducible and smooth. Indeed, $M$ can be covered by a
finite number of affine subsets $M_i$, and we can consider the
families ${\cal F}_{|M_i}$ separately. For the same reason we can
assume that $M$ is irreducible. Finally dim $M\setminus Reg(M)<$
dim $M$ and we can use induction to reduce the general case to the
smooth one.

Assume that $M$ is smooth and affine. If $\mu({\cal F})=0$, then
$\cal F$ does not contain any proper mapping. Hence we can assume
that $\mu({\cal F})=k>0.$ By Lemma \ref{gl'}  there is a non-empty
open subset $U\subset M$ such that for every $m_1, m_2\in U$ we
have

1) $\mu(f_{m_1})=\mu(f_{m_2})=k,$

2) the pairs $(\overline{f_{m_1}(X)}, B(f_{m_1}))$ and
$(\overline{f_{m_2}(X)}, B(f_{m_2}))$  are equivalent by means of
a homeomorphism, i.e., there is a homeomorphism $\Psi : Y\to Y$
such that $\Psi(\overline{f_{m_1}(X)})=\overline{f_{m_2}(X)}$ and
$\Psi(B(f_{m_1}))=B(f_{m_2}).$

 Fix a pair $Q=\overline{f_{m_0}(X)}, B=B(f_{m_0})$ for some $m_0\in U_3.$ For $m\in
U_3$ the   mapping $f_m: X\to Y$ is topologically equivalent to
the continuous mapping $f'_m=\Psi_m\circ f_m$ with
$\overline{f'_{m}(X)}=Q$ and $B(f'_m)=B$ (Lemma \ref{gl'}). Every
mapping $f'_m$  induces a topological covering $f'_m :X\setminus
{f'_m}^{-1}(B)=P_{f'_m}\to R=Q\setminus B$. Take a point $a\in R$
and let $a_{f'_m}\in {f'_m}^{-1}(a).$ We have an induced
homomorphism
$$f_* :\pi_1(P_{f'_m}, a_{f'_m})\to \pi_1(R,a).$$
Denote $H_f=f_*(\pi_1(P_f, a_f))$ and $G=\pi_1(R,a).$ Hence
$[G:H_f]=k.$ It is well known that the fundamental group of a
smooth algebraic variety  is finitely generated. In particular the
group $G:=\pi_1(Q\setminus B, a)$ is finitely generated. Let us
recall the following result of M. Hall (see  \cite{hal}):

\begin{lem}\label{hall}
Let $G$ be a finitely generated group and let $k$ be a natural
number. Then there are only a finite number of subgroups $H\subset
G$ such that $[G:H]=k.$
\end{lem}

By Lemma \ref{hall}  there are only a finite number of subgroups
$H_1,..., H_r\subset G$ with index $k.$ Choose proper  mappings
$f_i=f'_{m_i}=\Psi_i\circ f_{m_i} : X\to Y$ such that
$H_{f_i}=H_i$ (of course only if such a mapping $f_i$  does
exist). We show that every proper  mapping $f'_m$ ($m\in U$) is
equivalent to one of mappings $f_i.$

Indeed, let $H_{f'_m}=H_{f_i}$ (here $f'_m=\Psi_m\circ f_m$). We
show that $f'_m:=f$ is equivalent to $f_i.$ Let us consider two
coverings $f: (P_f, a_f)\to (R,a)$ and $f_i: (P_{f_i}, a_{f_i})\to
(R,a).$ Since $f_*(\pi_1(P_f, a_f))={f_i}_*(\pi_1(P_{f_i},
a_{f_i}))$ we can lift the covering $f$ to a homeomorphism $\phi:
P_f\to P_{f_i}$ such that following diagram commutes:
\begin{center}
\begin{picture}(240,120)(-40,40)
\put(180,160){\makebox(0,0)[tl]{$(P_{f_i}, a_{f_i})$}}
\put(20,40){\makebox(0,0)[tl]{$(P_f,a_f)$}}
\put(180,40){\makebox(0,0)[tl]{$(R,a)$}}
\put(190,100){\makebox(0,0)[tl]{$f_i$}}
\put(95,50){\makebox(0,0)[tl]{$f$}}
\put(80,100){\makebox(0,0)[tl]{$\phi$}}
\put(65,35){\vector(1,0){110}} \put(40,45){\vector(4,3){130}}
\put(183,145){\vector(0,-1){100}}
\end{picture}
\end{center}
\vspace{15mm}

\noindent Since the mappings $f$ and $f_i$ are proper, the mapping
$\phi$ can be extended to a continuous mapping $\Phi$ on the whole
of $X.$ Indeed, take a point $x\in f^{-1}(B)$ and let $y=f(x).$
The set $f_i^{-1}(y)=\{ b_1,..., b_s\}$ is finite. Take small open
disjoint neighborhoods $W_i(r)$ of  $b_i$, such that  $W_i(r)$
shrinks to $b_i$ as $r$ tends to $0.$ We can choose an
 open neighborhood $V(r)$ of $y$ so small that
$f_i^{-1}(V(r))\subset \bigcup^s_{j=1} W_i(r).$  Now take a small
connected neighborhood $P_x(r)$ of $x$ such that $f(P_x(r))\subset
V(r).$ The set $P_x(r)\setminus f^{-1}(B)$ is still connected and
it is transformed by $\phi$ into one particular set $W_{i_0}(r).$
We  take $\Phi(x)=b_{i_0}.$ It is easy to see that the mapping
$\Phi$ so defined is a continuous extension of $\phi.$ In fact
$\phi(P_x(r)\setminus f^{-1}(B))$ shrinks to $b_{i_0}$ if $r$ goes
to $0.$ Moreover, we still have $f=f_i\circ \Phi.$

In a similar way the mapping $\Lambda$ determined by $\phi^{-1}$
is  continuous. It is easy to see that $\Lambda\circ\Phi=\Phi\circ
\Lambda=identity$, hence $\Phi$ is a homeomorphism. Consequently,
the mapping $f_i\circ \Phi=\Psi_i\circ f_{m_i}\circ \Phi$ is equal
to $f=\Psi_m\circ f_{m}.$ Finally,  we get
$$(\Psi_i)^{-1}\circ\Psi_m\circ f_m\circ \Phi^{-1}=f_{m_i}.$$
This means that the family ${\cal F}_{|U}$ contains only a finite
number of topologically non-equivalent proper mappings. In fact,
the number of topological types of proper mappings in ${\cal
F}_{|U}$ is bounded by the number of subgroups of $G$ of index
$\mu({\cal F}).$

Let $T=M\setminus U.$ Hence dim $T<$ dim $M.$ By the  induction
the family ${\cal F}_{|T}$ also contains only a finite number of
topologically non-equivalent proper mappings. Consequently so does
${\cal F}.$
\end{proof}

\begin{co}
There is only a finite number of topologically non-equivalent
proper polynomial mappings $f: \C^n\to\C^m$ of bounded (algebraic)
degree.  $\square$
\end{co}

\section{Families of proper mappings}

In this section we extend our previous result a little in the case
of families of proper mappings. First we prove a following lemma:

\begin{lem}\label{lemat}
Let $Y=\R^{n}$ and let $Z\subset X$ be a linear subspace of $Y.$
 Fix
$\eta>0$ and let $B( 0, \eta)$ be a ball of radius $\eta.$  Let
$\gamma : I\ni t\mapsto \gamma(t)\in  B( 0, \eta)\cap Y$ be a
smooth curve. Take $\epsilon>\eta.$ Then there exists a continuous
family of diffeomorphisms $\Phi_t : X\to X$, $t\in [0,1]$ such
that

1) $\Phi_1(\gamma(t))=\gamma(0) \ {\rm {\it and}}\ \ \Phi_t(z)=z \
\ {\rm {\it for}}\ \ \|z\| \ge \epsilon.$

2) $\Phi_0={\rm identity}.$

3) $\Phi_t(Z)=Z.$
\end{lem}

\begin{proof}
Let $v_t=\gamma(t)-\gamma(0).$ Let $\sigma: Y\to [0, 1]$ be a
differentiable function such that $\sigma=1$ on $B(0, \eta)$ and
$\sigma=0$ outside $B(0, \epsilon).$ Define a vector field
$V(x)=\sigma(x)v_t.$ Integrating this vector field we get desired
diffeomeorphisms $\Phi_t.$
\end{proof}

\begin{co}\label{lem}
Let $Y$ be a smooth manifold  and $Z$ be a smooth submanifold. For
every point  $a\in  Z$ there is  an open connected subset $U_a$
such that if  $\gamma : I\ni t\mapsto \gamma(t)\in U\cap Z$ is a
smooth curve, then there is a continuous family of diffeomorphism
$\psi_t : Y\to Y$, $t\in [0,1]$ such that

1) $\psi_t(\gamma(t))= \gamma(0),$

2) $\psi_t(x)=x$ for $x\not\in U$ and $\Phi_0={\rm identity},$

3) $\psi_t(Z)=Z.$
\end{co}

Now we are in a position to prove:
\newpage

\begin{theo}\label{par}
Let $X,Y$ be smooth affine irreducible varieties. Let $\cal F:
M\times X\to Y$ be an algebraic family of proper polynomial
mappings from $X$ to $Y.$ Assume that $M$ is an irreducible
variety. Then there exists a Zariski open dense subset $U\subset
M$ such that for every $m,m'\in U$ mappings $f_m$ and $f_{m'}$ are
topologically equivalent.
\end{theo}

\begin{proof}
We follow the proof of Theorem \ref{gl1}. By Lemma \ref{gl'}
there is a non-empty open subset $U\subset M$ such that for every
$m_1, m_2\in U$ we have

1) $\mu(f_{m_1})=\mu(f_{m_2})=k,$

2) the pairs $(\overline{f_{m_1}(X)}, B(f_{m_1}))$ and
$(\overline{f_{m_2}(X)}, B(f_{m_2}))$  are equivalent by means of
a homeomorphism, i.e., there is a homeomorphism $\Psi : Y\to Y$
such that $\Psi(\overline{f_{m_1}(X)})=\overline{f_{m_2}(X)}$ and
$\Psi(B(f_{m_1}))=B(f_{m_2}).$

 Fix a pair $Q=\overline{f_{m_0}(X)}, B=B(f_{m_0})$ for some $m_0\in U.$ For $m\in
U$ the   mappings $f_m$ and $f_{m_0}$ can be connected by a
continuous path $f_t, f_0=f_{m_0}, f_1=f_m.$ Moreover we have also
a continuous family of homeomorphisms $\Psi_t : Y\to Y$ such that
$\Psi_t(\overline{f_{t}(X)})=\overline{f_{0}(X)}$ and
$\Psi_(B(f_{t}))=B(f_{0}).$ It is enough to prove that mappings
$F_t=\Psi_t\circ f_t$ are locally (in the sense of parameter $t$)
equivalent.

\vspace{5mm}

1) {\it First step of the proof.} Let $C_t\subset X$ denotes the
preimage by $F_t$ of the set $B$ (in fact $C_t=f_t^{-1}(B(f_t)$)
and put $X_t=X\setminus  C_t.$ Assume that for all mappings $F_t$
there is a point $a\in (X\setminus \bigcup_{t\in I} C_t)$ such
that for all $t\in I$ we have $F_t(a)=b.$ Put $Q':=Q\setminus B.$

We have an induced homomorphism ${G_t}_*: \pi_1(X_t,a)\to
\pi_1(Q',b)$. We show that the subgroup ${F_t}_*(
\pi_1(X_t,a))\subset \pi_1(Q',b)$ does not depend on $t.$

Indeed let $\gamma_1,..., \gamma_s$ be generators of the group
$\pi_1(X_{t_0}, a).$ Let $U_i$ be an open relatively compact
neighborhoods of $\gamma_i$ such that $\overline{U_i}\cap
C_{t_0}=\emptyset.$ For sufficiently small number $\epsilon>0$ and
$t\in (t_0-\epsilon,t_0+\epsilon)$ we have $\overline{U_i}\cap
C_{t}=\emptyset.$ Let $t\in (t_0-\epsilon,t_0+\epsilon)$. Note
that the loop $F_t(\gamma_i)$ is homotopic with the loop
$F_{t_0}(\gamma_i).$ In particular the group ${F_{t_0}}_*(
\pi_1(X_{t_0},a))$ is contained in the group ${F_t}_*(
\pi_1(X_t,a)).$ Since they have the same index in  $\pi_1(Y',b)$
they are equal. This means that the subgroup ${G_t}_*(
\pi_1(X_t,a))\subset \pi_1(Y',b)$ is locally constant, hence it is
constant.

Let us consider two coverings $F_t: (X_t, a)\to (Q',b)$ and $F_0:
(X_0, a)\to (Q',b)$. Since ${F_t}_* \pi_1(X_t, a)={F_0}_*\pi_1
(X_0, a)$ we can lift the covering $F_t$ to a homeomorphism
$\phi_t: X_t\to X_0.$ As before we can extend the mapping $\phi_t$
to the mapping $\Phi_t: X\to X$ which satisfies all desired
conditions.

\vspace{5mm}

2) {\it The general case.} Now we can prove Theorem \ref{par}.
First we prove that for every $t_0\in I$ there exists $\epsilon>0$
and a  family of diffeomorphisms $\Phi_t: X\to X$, $t\in
(t_0-\epsilon, t_0+\epsilon)$ such that $F_t=F_{t_0}\circ \Phi_t$
for $t\in (t_0-\epsilon, t_0+\epsilon).$ Take a point $a\in
X_{t_0}$ and choose $\epsilon>0$ so small that $a\in X_t$ for
$t\in (t_0-\epsilon, t_0+\epsilon)$. Put $\gamma(t)\ni t\mapsto
F_t(a)\in Y'.$ We can take $\epsilon$ so small that the hypothesis
of Corollary \ref{lem} is satisfied. Applying Corollary \ref{lem}
with $Y'=Y\setminus B$ and $Z=Q\setminus B$ we have a continuous
family of diffeomeorphisms $\psi_t: Y\to Y$ which preserves $Q$
and $B$, $t\in (t_0-\epsilon, t_0+\epsilon)$ such that
$\psi_t(F_t(a))=F_{0}(a).$ Take $G_t=\psi_t\circ F_t.$ Arguing as
in the first part of our proof all $G_t$ are topologically
equivalent  for $t\in (t_0-\epsilon, t_0+\epsilon)$. Hence also
all $F_t$ are topologically equivalent for $t\in (t_0-\epsilon,
t_0+\epsilon)$. Since $F_t$ are locally topologically equivalent,
they are topologically equivalent for every $t\in I.$
\end{proof}

\begin{co}
Let $n\le m$ and let $\Omega_n(d_1,...,d_m)$ denotes the family of
all polynomial mappings $F=(f_1,...,f_m):\C^n\to\C^m$ of a fixed
multi-degree $(d_1,...,d_m).$ Then any two general member of this
family have the same topology.
\end{co}

\end{document}